\documentclass[reqno,11pt]{amsart}

\usepackage{amsmath,amssymb,amsthm,mathrsfs}
\usepackage{graphicx,epsfig,color}
\usepackage[latin1]{inputenc}

\usepackage[numbers]{natbib}

\numberwithin{equation}{section}

\def\H{\mathcal H}

\def\R{\mathbb R}
\def\N{\mathbb N}

\newcommand{\dist}{\mathop{\mathrm{dist}}}
\def\e{\varepsilon}

\def\S{\Sigma}

\def\l{\lambda}
\def\g{\gamma}

\def\de{\delta}

\def\co{{\rm co}\,}

\def\pa{\partial}

\def\00{{\bf 0}}

\renewcommand{\a}{\alpha}

\renewcommand{\l}{\lambda}

\newcommand{\ov}{\overline}

\newcommand{\diam}{\mathrm{diam}}
\newcommand{\cc}{\subset\subset}

\def\Lip{{\rm Lip}\,}

\newtheorem{thm}{Theorem}[section]

\newtheorem{lem}[thm]{Lemma}
\newtheorem{prop}[thm]{Proposition}
\newtheorem{rem}[thm]{Remark}
\theoremstyle{definition}
\newtheorem{defn}[thm]{Definition}

\begin{document}

\title[Sharp quantitative concentration inequality]{The sharp quantitative \\ Euclidean concentration inequality}

\author{Alessio Figalli}
\address{Department of Mathematics, UT Austin, Austin, TX 78712}
\email{\tt figalli@math.utexas.edu}

\author{Francesco Maggi}
\address{Department of Mathematics, UT Austin, Austin, TX 78712}
\email{\tt maggi@math.utexas.edu}

\author{Connor Mooney}
\address{Department of Mathematics, UT Austin, Austin, TX 78712}
\email{\tt cmooney@math.utexas.edu}


% ----------------------------------------------------------------
\begin{abstract}
The Euclidean concentration inequality states that, among
sets with fixed volume, balls have $r$-neighborhoods of minimal volume for every $r>0$.
On an arbitrary set, the deviation of this volume growth from that of a ball
is shown to control the square of the volume of the symmetric difference between the set
and a ball. This sharp result is strictly related to the physically significant problem of
understanding near maximizers in the Riesz rearrangement inequality with a strictly decreasing radially decreasing kernel.
Moreover, it implies as a particular case the sharp quantitative Euclidean isoperimetric inequality from \cite{fuscomaggipratelli}.
\end{abstract}

\maketitle
% ----------------------------------------------------------------

%%%%%%%%%%%%%%%%%%%%%%%%%%%%%%%%%%%%%%%%%%%%%%%%%%%%%%%%%%%%%%%%%%%%%%%%%%%%%%%%%%%%%%%%%%%%%%%%%%%%%%%%%%

\section{Introduction}
\subsection{Overview} A sharp stability theory for the Euclidean isoperimetric inequality has been established in recent years in several papers \cite{Fuglede,hall,fuscomaggipratelli,maggibams,FigalliMaggiPratelliINVENTIONES,CicaleseLeonardi,fuscojulin}, and has found applications to classical capillarity theory \cite{FigalliMaggiARMA,maggimihaila}, Gamow's model for atomic nuclei \cite{knupfermuratovI,knupfermuratovII}, Ohta-Kawasaki model for diblock copolymers \cite{cicalesespadaro,goldmanmuratovserfatyI,goldmanmuratovserfatyII}, cavitation in nonlinear elasticity \cite{henaoserfaty}, and slow motion for the Allen-Cahn equation \cite{murrayrinaldi}. Depending on the application, a sharp stability theory not only conveys better estimates, but it is actually crucial to conclude anything at all.

The goal of this paper is obtaining a sharp stability result for the Euclidean concentration inequality (i.e., among all sets, balls minimize the volume-growth of their $r$-neighborhoods). As in the case of the Euclidean isoperimetric inequality, obtaining a sharp stability theory for Euclidean concentration is very interesting from the geometric viewpoint. In addition to this, a strong motivation comes from physical applications. Indeed, as explained in detail in \cite[Section 1.4]{carlenmaggi}, the Euclidean concentration inequality
%is the particular case of the Brunn-Minkowski inequality that is obtained when one of the two involved sets is a ball. In turn, this instance of the Brunn-Minkowski inequality
is crucial in characterizing equality cases  for the physically ubiquitous Riesz rearrangement inequality
(see e.g. \cite[Theorem 3.9]{lieblossBOOK}) in the case of radially symmetric strictly decreasing kernels.

The need for a quantitative version of the Euclidean concentration inequality in the study of the Gates-Penrose-Lebowitz free energy functional arising in Statistical Mechanics \cite{lebowitzpenrose,gatespenrose} has been pointed out in \cite{CarlenCELM}. A recent progress has been achieved in \cite{carlenmaggi}, where a non-sharp quantitative version of Euclidean concentration is obtained (see below for more details), and where its application to a quantitative understanding of equality cases in Riesz rearrangement inequality is also discussed. The argument used in \cite{carlenmaggi}, however, is not precise enough to provide a {\it sharp} stability result, which is the objective of our paper. As a second example of physical applications, a quantitative understanding of the Riesz rearrangement inequality for the Coulomb energy arises in the study of dynamical stability of gaseous stars, and was recently addressed, again in non-sharp form, in \cite{burchardchambers}.

On the geometric side, we note that the Euclidean concentration inequality is equivalent to the Euclidean isoperimetric inequality. Going from isoperimetry to concentration is immediate by the coarea formula, while the opposite implication is obtained  by differentiation (see \eqref{argument} below). Despite this direct relation, the various methods developed so far in the study of sharp stability for the Euclidean isoperimetric inequality (symmetrization techniques \cite{hall,hallhaymanweitsman,fuscomaggipratelli}, mass transportation \cite{FigalliMaggiPratelliINVENTIONES}, spectral analysis and selection principles \cite{Fuglede,CicaleseLeonardi}) do not seem to be suitable for Euclidean concentration (see Section \ref{sect:relation} below). To obtain a sharp result we shall need to introduce several new geometric ideas, specific to this nonlocal setting, and to combine them with various auxiliary results, such as Federer's Steiner-type formula for sets of positive reach and the strong version of the sharp quantitative isoperimetric inequality \cite{fuscojulin}.

\subsection{Statement of the main theorem} If $E$ is a subset of $\R^n$ with (outer) Lebesgue measure $|E|$, and denote by
\[
I_r(E):=\big\{x\in\R^n:\dist(x,E)<r\big\}\qquad r>0\,,
\]
the $r$-neighborhood of $E$. Then the {\it Euclidean concentration inequality} asserts that
\begin{equation}
  \label{euclidean concentration}
  |I_r(E)|\ge|I_r(B_{r_E})|\,,\qquad\forall \,r>0\,,
\end{equation}
where
\[
r_E:=\bigg(\frac{|E|}{|B_1|}\bigg)^{1/n}=\mbox{radius of a ball of volume $|E|$}\,.
\]
It is well-known that inequality \eqref{euclidean concentration} is a non-local generalization
of the classical {\it Euclidean isoperimetric inequality}: the latter states that, for any set $E$,
its distributional perimeter $P(E)$ is larger than the one of the ball $B_{r_E}$, namely
\begin{equation}
  \label{isoperimetric inequality}
  P(E)\ge P(B_{r_E})\,.
\end{equation}
The isoperimetric inequality \eqref{isoperimetric inequality} can be deduced from \eqref{euclidean concentration}
as follows: if $E$ is a smooth bounded set then, using \eqref{euclidean concentration},
\begin{equation}
  \label{argument}
P(E)=\lim_{r\to 0^+} \frac{|I_r(E)|-|E|}{r} \geq \lim_{r\to 0^+} \frac{|I_r(B_{r_E})|-|E|}{r}=P(B_{r_E}).
\end{equation}
Then, once \eqref{isoperimetric inequality} has been obtained on smooth set,
the general case follows by approximation.

We also mention (although we shall not investigate this here) that the Gaussian counterpart of \eqref{euclidean concentration} plays a very important role in Probability theory, see \cite{ledoux}.\\

It turns out that if equality holds in \eqref{euclidean concentration} for some $r>0$ then, modulo a set of measure zero, $E$ is a ball (the converse is of course trivial). The stability problem amounts in quantifying the degree of sphericity possed by almost-equality cases in \eqref{euclidean concentration}. This problem is conveniently formulated by introducing the following two quantities
\begin{eqnarray}
  &&\delta_r(E):=\max\bigg\{\frac{r}{r_E},\frac{r_E}{r}\bigg\}\,\bigg(\frac{|I_r(E)|}{|I_r(B_{r_E})|}-1\bigg)\,,\qquad r>0\,,
  \\
  &&\a(E):=\inf\bigg\{\frac{|E\Delta B_{r_E}(x)|}{|B_{r_E}|}:x\in\R^n\bigg\}\,,\qquad B_s(x):=x+B_s\,.
\end{eqnarray}
Notice that $\a(\l\,E)=\a(E)$ and $\de_{\l\,r}(\l\,E)=\de_r(E)$ for every $r,\l>0$ and $E\subset\R^n$, with $\de_r(E)=0$ if and only if $E$ is a ball (up to a set of measure zero). The quantity $\a(E)$ is usually called the {\it Fraenkel asymmetry of $E$}. We now state our main theorem.

\begin{thm}\label{main}
 For every $n\ge 2$ there exists a constant $C(n)$ such that
 \begin{equation}
   \label{main inequality}
    \alpha(E)^2 \leq C(n)\,\delta_r(E)\,,\qquad\forall \,r>0\,,
 \end{equation}
 whenever $E$ is a measurable set with $0<|E|<\infty$; more explicitly, there always exists $x\in\R^n$ such that
 \[
 |I_r(E)|\ge|I_r(B_{r_E})|\bigg\{1+c(n)\,\min\bigg\{\frac{r}{r_E},\frac{r_E}{r}\bigg\}\,\bigg(\frac{|E\Delta B_{r_E}(x)|}{|E|}\bigg)^2\bigg\}\,,
 \]
 for some positive constant $c(n)$.
\end{thm}

\begin{rem}
  {\rm Notice that
\[
|I_r(E)|=r^n|B_1|+{\rm O}(r^{n-1})\quad\mbox{as $r\to\infty$}\,,\qquad |I_r(E)|=|E|+r\,(P(E)+{\rm o}(1))\quad\mbox{as $r\to 0^+$}
\]
(on sufficiently regular sets), therefore the factor $\max\{r/r_E,r_E/r\}$ in the definition of $\de_r(E)$ is needed to control $\a(E)$ with a constant independent of $r$.}
\end{rem}

\begin{rem}
  {\rm The decay rate of $\a(E)$ in terms of $\de_r(E)$ is optimal. Take for example an ellipse which is a small deformation of $B_1$ given by
$$E_\epsilon := \text{diag}\bigg(1+\epsilon,\, \frac{1}{1+\epsilon}, \,1, \,\ldots,\,1\bigg)B_1.$$
Then $\alpha(E_\epsilon)$ is order $\epsilon$, and one computes that
$$|I_r(E_\epsilon)| = |B_{1+r}|\left(1 + O(\epsilon^2)\,\frac{r}{(1+r)^2}\right),$$
giving $\delta_r(E_\epsilon) = O(\epsilon^2)$.}
\end{rem}

\begin{rem}
  {\rm When $n=1$ \eqref{main inequality} holds with $\a(E)$ in place of $\a(E)^2$, see \cite[Theorem 1.1]{carlenmaggi}
  or \cite[Theorem 1.1]{figallijerison2}.}
\end{rem}

\subsection{Relations with other stability problems}
\label{sect:relation}
Theorem \ref{main} is closely related to the sharp quantitative isoperimetric inequality from \cite{fuscomaggipratelli}
\begin{equation}
  \label{sharp quantitative isoperimetric inequality}
  \de_{{\rm iso}}(G)\ge c(n)\,\a(G)^2\qquad\mbox{if $0<|G|<\infty$}\,,
\end{equation}
where
\[
\de_{{\rm iso}}(G):=\frac{P(G)}{n|B_1|^{1/n}|G|^{(n-1)/n}}-1\,.
\]
This result is used in the proof of \eqref{main inequality},
which in turns implies   \eqref{sharp quantitative isoperimetric inequality} in the limit $r\to 0^+.$
However, the proof of Theorem \ref{main} requires entirely new ideas with respect to the ones developed so far in
the study of \eqref{sharp quantitative isoperimetric inequality}.
Indeed, the original proof of \eqref{sharp quantitative isoperimetric inequality} in \cite{fuscomaggipratelli}
uses symmetrization techniques and an induction on dimension argument, based on slicing, that it is unlikely to yield the sharp exponent
and the independence from $r$ in the final estimate (compare also with \cite{figallijerison1,figallijerison2}).
The mass transportation approach adopted in \cite{FigalliMaggiPratelliINVENTIONES} requires a regularization
step that replaces a generic set with small $\delta_{\rm iso}$ by a nicer set on which the trace-Poincar\'e inequality holds.
While, in our case, it is possible to regularize a set with small $\delta_r$ above a scale $r$ (and indeed this idea is used in the proof of Theorem \ref{main}),
it is unclear to us how to perform a full regularization below the scale $r$.
Finally, the penalization technique introduced in \cite{CicaleseLeonardi} is based on the regularity theory for local almost-minimizers
of the perimeter functional, and there is no analogous theory in this context.

In addition to extending the sharp quantitative isoperimetric inequality, Theorem \ref{main}
marks a definite progress in the important open problem of proving a sharp quantitative stability result for the {\it Brunn-Minkowski
inequality} (see \cite{gardner} for an exhaustive survey)
\begin{equation}\label{BM}
|E + F|^{1/n} \geq |E|^{1/n} + |F|^{1/n}\,,
\end{equation}
where $E+F=\{x+y:x\in E\,,\ y\in F\}$ is the {\it Minkowski sum} of $E,F\subset\R^n$. Let us recall that equality holds in \eqref{BM} if and only if both $E$ and $F$ are convex, and one is a dilated translation of the other. The stability problem for \eqref{BM} can be formulated in terms of the
 {\it Brunn-Minkowski deficit $\delta(E,F)$ of $E$ and $F$}
\[
\delta(E,F) := \max\left\{\frac{|E|}{|F|},\,\frac{|F|}{|E|}\right\}^{1/n}\left(\frac{|E + F|^{1/n}}{|E|^{1/n} + |F|^{1/n}} - 1\right)\,,
\]
(which is invariant by scaling both $E$ and $F$ by a same factor) and of the {\it relative asymmetry of $E$ with respect to $F$}
\[
\alpha(E,F) := \inf\bigg\{\frac{|E\, \Delta \, (x+ \lambda\, \co(F))|}{|E|}:x\in\R^n\bigg\}\qquad \l=\Big(\frac{|E|}{|\co(F)|}\Big)^{1/n}
\]
(which is invariant under possibly different dilations of $E$ and $F$). In the case that both $E$ and $F$ in \eqref{BM} are convex, the optimal stability result
\begin{equation}
\label{stability BM E F sharp}
\alpha(E,F)^2 \leq C(n)\,\delta(E,\,F)
\end{equation}
was established, by two different mass transportation arguments,
in \cite{figallimaggipratelliBrunnMink,FigalliMaggiPratelliINVENTIONES}.
Then, in \cite[Theorem 1.1]{carlenmaggi}, exploiting \eqref{sharp quantitative isoperimetric inequality}, the authors prove that if $E$ is arbitrary and $F$ is convex then
\begin{equation}
  \label{carlenmaggi thm}
  \a(E,F)^4\le C(n)\,\max\bigg\{1,\frac{|F|}{|E|}\bigg\}^{(4n+2)/n}\,\de(E,F)\,.
\end{equation}
In the general case when $E$ and $F$ are arbitrary sets, the best results to date
are contained in \cite{christ,figallijerison1,figallijerison2}, where, roughly speaking, $\alpha(E,\,F)^2$ is replaced by
$\alpha(E,\,F)^{\eta(n)}$ for $\eta(n)$ explicit, and the inequality degenerates when $|F|/|E|$ approaches $0^+$ or $+\infty$. Despite the considerable effort needed to obtain such result -- whose proof is based on measure theory, affine geometry, and (even in the one-dimensional case!) additive combinatorics -- the outcome is still quite far from addressing the natural conjecture that, even on arbitrary sets, the optimal exponent for $\a(E,F)$ should be $2$, with no degenerating pre-factors depending on volume ratios. Since
\[
I_r(E)=E+B_r\qquad\forall r>0\,,
\]
Theorem \ref{main} solves this challenging conjecture when $E$ is arbitrary and $F=B_r$ is a ball. This case is definitely one of the most relevant from the point of view of physical applications.
%%Notice also that Theorem \ref{main} improves the case $F=B_r$ of \eqref{carlenmaggi thm} in two directions, in the sense that $\a(E)^4$ is replaced by the sharp $\a(E)^2$, and $C(n)\,\max\left\{1,|F|/{|E|}\right\}^{(4n+2)/n}$ is replaced by $C(n)$. As it will be apparent from the description of the proof, obtaining both improvements is quite non-trivial.

\subsection{Description of the proof of Theorem \ref{main}} The starting point, as in \cite{carlenmaggi}, is the idea
of using the coarea formula to show that
\begin{equation}
  \label{coarea approx}
  \de_r(E)  \text{ is equal to a certain average of } \g_E(s) :=P(E+B_s)-P(B_{1+s})\,
\end{equation}
over the interval $0 < s < r$. Observe that
\[
\g_E(s)\ge c(n)\,|E+B_s|^{(n-1)/n}\,\de_{{\rm iso}}(E+B_s)\,
\]
thanks to \eqref{euclidean concentration}. When $\de_{{\rm iso}}(E+B_s)$ is comparable to $\de_{{\rm iso}}(E)$ for a substantial set of
radii $s\in(0,r)$, then \eqref{sharp quantitative isoperimetric inequality} would allow us to conclude $\de_r(E)\ge c(n)\,\a(E)^2$.
The difficulty, however, is that $\de_{{\rm iso}}(E+B_s)$ may decrease very rapidly for $s$ close to $0$, for example if $E$ has lots of small ``holes'', so that
$\de_{{\rm iso}}(E + B_s)$ becomes negligible with respect to $\de_{{\rm iso}}(E)$. This problem is addressed in \cite{carlenmaggi} by
an argument which is not precise enough to obtain a quadratic decay rate in the final estimate.

The first key idea that we exploit to overcome this difficulty is introducing a suitable regularization procedure which is compatible with a reduction argument on $\de_r(E)$ and $\a(E)$. More precisely, we introduce the {\it $r$-envelope $\co_r(E)$ of a bounded set $E$}
\[
\co_r(E):=\bigcap\Big\{B_r(x)^c:B_r(x)\subset E^c\Big\},
\]
where, given a set $A$, we use the notation $A^c:=\R^n\setminus A$. From the geometric point of view, it is convenient to
think of $\co_r(E)$ as the set obtained by sliding balls of radius $r$ from the exterior of $E$ until they touch $E$, and filling in the holes of
size $r$ or smaller that remain. This construction regularizes the boundary of $E$ at scale $r$,
while not changing the parallel surface at distance $r$ from $E$ (see Figure \ref{EnvelopePic}).

\begin{figure}
 \centering
    \includegraphics[scale=0.35]{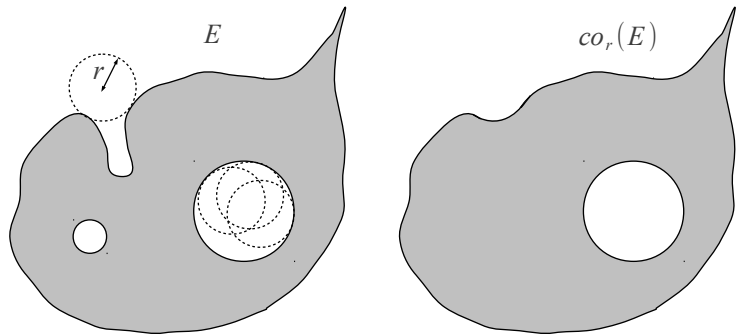}
 \caption{{\small Taking the $r$-envelope $co_r(E)$ regularizes the boundary at scale $r$, without changing the sumset $E + B_r$.}}
\label{EnvelopePic}
\end{figure}

This intuition can be made quantitative, in the sense that $\co_r(E)$ always satisfies an exterior ball condition of radius $r$ (Lemma \ref{lemma EnvelopeProps}) and
has a universal perimeter upper bound in each ball of radius $s$ comparable to $r$ (see Lemma \ref{lemma AreaBound}). Moreover, since $E+B_r=\co_r(E)+B_r$, then we easily see that, up to
excluding some trivial situations, one always has
\[
\a(E)\le C(n)\,\a(\co_r(E))\,,\qquad \de_r(\co_r(E))\le C(n)\,\de_r(E)\,,
\]
see Proposition \ref{prop EnvelopeReduction}. Hence, in the proof of Theorem \ref{main}, one can assume that $E=\co_r(E)$. We call such sets {\it $r$-convex}.

The reduction to $r$-convex sets is particularly effective when $r$ is large with respect to $r_E$.
A first remark is that if $r/r_E$ is large enough and $\de_r(E)\le 1$, then $E$ has bounded diameter (in terms of $n$), after dilating
so that $r_E = 1$. This is in sharp contrast with the situation met in the quantitative study of the isoperimetric inequality, where long spikes of small volume and perimeter are of course compatible with the smallness of the isoperimetric deficit $\de_{{\rm iso}}$. The $r$-convexity can then be used jointly with John's lemma to show that if $E\subset B_{R(n)}$, then $B_{r(n)}\subset E$ for some positive dimensional constant $r(n)$. Starting from these properties we can actually check that $\pa E$ is a radial Lipschitz graph with respect to the origin, and that it has positive reach $1$. This last property means that every point at distance at most $1$ from $E$ has a unique projection on $E$. By a classical result of Federer, see Theorem \ref{SteinerFormula} below, it follows that $P(E+B_s)$ is a polynomial of degree at most $(n-1)$ for $s\in[0,1]$. In particular, $\g_E(s)$ is a non-negative polynomial of degree at most $(n-1)$ for $s\in[0,1]$, so that
\[
|I_1(E)| - |I_1(B_1)| = \int_0^1\g_E(s)\,ds\ge c(n)\,\g_E(0)
\]
by a compactness argument (see Lemma \ref{PolyLem}). Combining this bound with
the fact that $\g_E(0)$ controls $\de_{{\rm iso}}(E)$, \eqref{sharp quantitative isoperimetric inequality}, and the Brunn-Minkowski inequality,
we are able to complete the proof in the regime when $r/r_E$ is large.

We complement the above argument with a different reasoning, which is effective when $r/r_E$ is bounded from above, but actually {degenerates as $r/r_E$ becomes larger}. The key tool here is the strong quantitative version of the isoperimetric inequality from \cite{fuscojulin}, whose statement is now briefly recalled. Let us set
\begin{equation}
  \label{beta definition}
\beta(G)^2:= \inf\bigg\{\frac1{|G|^{(n-1)/n}}\int_{\pa^*G}\Big|\nu_G(y)-\frac{y-x}{|y-x|}\Big|^2\,d\mathcal{H}^{n-1}_y:x\in\R^n\bigg\}\,,
\end{equation}
\begin{equation}
  \label{betastar definition}
  \beta^*(G)^2:=\inf\bigg\{\bigg(\frac{|G\Delta B_{r_G}(x)|}{|G|}\bigg)^2+\frac1{|G|^{(n-1)/n}}\int_{\pa^*G}\Big|\nu_G(y)-\frac{y-x}{|y-x|}\Big|^2\,d\mathcal{H}^{n-1}_y:x\in\R^n\bigg\}\,,
\end{equation}
where $\nu_G$ and $\pa^*G$ denote, respectively, the (measure theoretic) outer unit normal and the reduced boundary of $G$ (if $G$ is an open sets with $C^1$-boundary, $\nu_G$ is standard notion of outer unit normal to $G$, and $\pa^*G$ agrees with the topological boundary of $G$). The quantity $\beta(G)$ measures the $L^2$-oscillation of $\nu_G$ with respect to that of a ball, and in \cite{fuscojulin} it is proved that
\begin{equation}\label{FJIsoperimetric}
 C(n)\,\de_{{\rm iso}}(G)\ge \beta^*(G)^2\,,\qquad 0<|G|<\infty\,.
\end{equation}
Clearly $\a(G)\le\beta^*(G)$, so \eqref{FJIsoperimetric} implies \eqref{sharp quantitative isoperimetric inequality}.
Moreover, as shown in \cite[Proposition 1.2]{fuscojulin}, the quantities $\beta(G)$ and $\beta^*(G)$
are actually equivalent:
\begin{equation}\label{OscVsAsymmetry}
\beta(G)\leq \beta^*(G)\le C(n)\,\beta(G)\,.
\end{equation}
Our argument (based on \eqref{FJIsoperimetric}) is then the following. By combining a precise form of \eqref{coarea approx} with
\eqref{FJIsoperimetric} we deduce that, for some $s<r$, we have
\[
C(n,r/r_E)\,\de_r(E)\ge\beta(E+B_s)^2\,(|E + B_s|/|E|)^{(n-1)/n}\,,
\]
where $C(n,r/r_E)$ is a constant that degenerates for $r/r_E$ large.
We now exploit the $r$-convexity property to apply the area formula between the surfaces $\pa E$ and $\pa(E+B_s)$ and compare $\beta(E+B_s)^2(|E+B_s|/|E|)^{(n-1)/n}$ with $\beta(E)^2$,
and thus with $\a(E)^2$, which concludes the proof of Theorem \ref{main}.\\
%It seems that removing the degeneracy in this estimate for $r/r_E$ large would require geometric reductions similar to those we describe above, which
%allow us to avoid using the strong quantitative isoperimetric inequality anyways.

\subsection{Organization of the paper} In section \ref{section reduction} we introduce the notion of $r$-envelope, show some key properties of $r$-convex sets, and reduce the proof of Theorem \ref{main} to this last class of sets. In section \ref{section small} we present the argument based on the strong form of the quantitative isoperimetric inequality, while in section \ref{section large} we address the regime when $r/r_E$ is large.

\section{Reduction to $r$-convex sets}\label{section reduction} In this section we introduce a geometric regularization procedure for
subsets of $\R^n$ which can be effectively used to reduce the class of sets considered in Theorem \ref{main}.
We recall the notation $B_r(x)=\{y\in\R^n:|y-x|<r\}$ for the ball of center
$x\in\R^n$ and radius $r>0$ (so that $B_r=B_r(0)$) and $E^c=\R^n\setminus E$ for the complement of $E\subset\R^n$.

We begin
with the following definitions.
\begin{defn}\label{ExtBallCondition}
 Let $E \subset \mathbb{R}^n$ and $r>0$.

 \noindent (i) $E$ satisfies the {\it exterior ball condition of radius $r$} if, for all $y \in \partial E$, there exists
 some ball $B_r(x) \subset E^c$ such that $y \in \partial B_r(x)$.

 \noindent (ii) the {\it $r$-envelope of $E$} is the set $\co_r(E)$ defined by
 \[
 \co_r(E) = \Big(\bigcup\big\{B_r(x):B_r(x)\subset E^c\big\}\Big)^c=\bigcap\Big\{B_r(x)^c:E\subset B_r(x)^c\Big\}\,.
 \]
 We say that $E$ is {\it $r$-convex} if $E=\co_r(E)$.
\end{defn}

%%%% Note that for $r = \infty$ this is just the convex hull.
%% I took this remark out because we cannot really plug $r=\infty$ in the definition

Notice that $E$ is convex if and only if it satisfies the exterior ball condition of radius $r$ for every $r>0$. Similarly
$E$ is convex if and only if $E=\co_r(E)$ for every $r>0$.

In Lemma \ref{lemma EnvelopeProps} we collect some useful properties of $\co_r(E)$, which are then exploited in Proposition
\ref{prop EnvelopeReduction} to reduce the proof of Theorem \ref{main} to the case of $r$-convex sets. Then, in Lemma
\ref{lemma AreaBound} we prove a uniform upper perimeter estimate for $r$-envelopes which will play an important role in
the proof of Theorem \ref{main}.

\begin{lem}\label{lemma EnvelopeProps}
 If $E\subset\R^n$ is bounded and $r>0$, then:
 \begin{enumerate}
    \item[(i)] $E\subset\co_r(E)$;
    \item[(ii)] $\co_r(E)$ is compact;
    \item[(iii)] $B_r(x)\subset E^c$ if and only if $B_r(x)\subset\co_r(E)^c$;
    \item[(iv)] $E+B_r=\co_r(E)+B_r$;
    \item[(v)] $\co_r(E)$ satisfies the exterior ball condition of radius $r$.
  \end{enumerate}
\end{lem}

\begin{proof}
 The first three conclusions are immediate.

 If $x\in \co_r(E)+B_r$, then there exists $y\in \co_r(E)$
 such that $y\in B_r(x)$; in particular $B_r(x)\cap \co_r(E)\ne\emptyset$, hence $B_r(x)\cap E\ne\emptyset$ by (iii), and thus $x\in E+B_r$, which proves (iv).

 Finally, given $x \in \partial( \co_r(E))$, let $x_j\to x$ with $x_j\in\co_r(E)^c$, so that there exists $y_j$ with $E\subset B_r(y_j)^c$ and $x_j\not\in B_r(y_j)^c$. Since $|x_j-y_j|< r$ and $x_j\to x$,
 up to a subsequence we can find $y \in \R^n$ such that $y_j\to y$ with $|x-y|\le r$. On the other hand, since $E\subset B_r(y_j)^c$, it follows by (iii) that $B_r(y_j)\subset\co_r(E)^c$, thus $B_r(y)\subset\co_r(E)^c$. As $x\in\co_r(E)$, this implies that $|x-y|\ge r$. In conclusion $|x-y|=r$ and $B_r(y)\subset\co_r(E)^c$, so (v) is proved.
\end{proof}

We now address the reduction to $r$-convex sets. Let us recall the following elementary property of the Fraenkel asymmetry:
\begin{equation}
  \label{asymmetry lipschitz}
  \big||E|\,\a(E)-|F|\,\a(F)\big|\le |E\Delta F|\,,\qquad 0<|E|\,|F|<\infty\,,
\end{equation}
see e.g. \cite[Lemma 2.2]{carlenmaggi}.

\begin{prop}\label{prop EnvelopeReduction}
Let $E$ be a bounded measurable set and $r>0$. There exists a dimensional constant $C(n)$ such that:
 \begin{itemize}
 \item[(a)] either $C(n)\,\de_r(E)\ge \a(E)^2$;
 \item[(b)] or
 \begin{equation}
   \label{reduction}
    \a(E)\le C(n)\,\a(\co_r(E))\,,\qquad \de_r(\co_r(E))\le C(n)\,\de_r(E)\,.
 \end{equation}
 \end{itemize}
\end{prop}

\begin{proof} Without loss of generality we can assume that $|E|=|B_1|$, so that $r_E=1$.
 For a suitable constant $b(n)>0$, we
split the argument depending on whether $|\co_r(E)\setminus E|\ge b(n)\,\a(E)$ or not.

If $|\co_r(E)\setminus E|\ge b(n)\,\a(E)$ then, since $\alpha(E) < 2$ and $|E|=|B_1|$, we have
 \[
 |\co_r(E)|^{1/n}\geq \bigl(|B_1|+b(n)\,\a(E)\bigr)^{1/n}\ge|B_1|^{1/n} + c(n)b(n)\alpha(E)\,.
 \]
 Thus, applying Lemma \ref{lemma EnvelopeProps}-(iv) and the Brunn-Minkowski inequality
 we get
 \begin{align*}
 \delta_r(E) &\ge \max\bigg\{r,\frac1r\bigg\}\left(\frac{|E + B_r|^{1/n}}{(1+r)|B_1|^{1/n}}-1\right) \\
 &= \max\bigg\{r,\frac1r\bigg\}\left(\frac{|\co_r(E) + B_r|^{1/n}}{|\co_r(E)|^{1/n} + r|B_1|^{1/n}}\frac{|\co_r(E)|^{1/n} + r|B_1|^{1/n}}{(1+r)|B_1|^{1/n}} - 1\right) \\
 &\geq c(n)\max\bigg\{r,\frac1r\bigg\}\,\frac{b(n)\alpha(E)}{1+r}
 \geq c(n)\,\alpha(E)\ge c(n)\,\a(E)^2\,,
 \end{align*}
 where in the last inequality we have used again $\a(E)<2$.
This proves the validity of (a).

 We are thus left to show that if $|\co_r(E)\setminus E|< b(n)\,\a(E)$, then \eqref{reduction} holds. By exploiting \eqref{asymmetry lipschitz}
 with $F=\co_r(E)$ and assuming $b(n)$ small enough, then $|\co_r(E)\setminus E|< b(n)\,\a(E)$ gives $\a(\co_r(E))\ge c(n)\a(E)$. At the same
 time, using again that $|\co_r(E)\setminus E|< b(n)\,\a(E) \leq 2\,b(n)$ we find
 that the volumes of $|E|$ and $|\co_r(E)|$ are comparable,
 therefore
 \[
 \max\bigg\{\frac{|\co_r(E)|}{|B_r|},\frac{|B_r|}{|\co_r(E)|}\bigg\}\le C(n)
  \max\bigg\{\frac{|E|}{|B_r|},\frac{|B_r|}{|E|}\bigg\}\,.
 \]
In addition, it follows by Lemma \ref{lemma EnvelopeProps}-(iv) and the trivial inequality $|E|\le|\co_r(E)|$ that
 \[
 \frac{|\co_r(E)+B_r|^{1/n}}{|\co_r(E)|^{1/n}+|B_r|^{1/n}}\le\frac{|E+B_r|^{1/n}}{|E|^{1/n}+|B_r|^{1/n}}\,.
 \]
This shows that $\de_r(\co_r(E))\le C(n)\,\de_r(E)$, and \eqref{reduction} is proved.
\end{proof}

Intuitively, $r$-convex sets are nice up to scale $r$. In this direction, the following lemma provides a uniform perimeter bound. Here and in the sequel we use the notation $P(F;B_s(x))$ to denote the perimeter of a set $F$ inside the ball $B_s(x)$. In particular, if $F$ is a smooth set, $P(F;B_s(x))=\H^{n-1}(\pa F\cap B_s(x))$.

\begin{lem}\label{lemma AreaBound}
If $E$ is a bounded set and $r>0$, then $\co_r(E)$ has finite perimeter and
\begin{equation}
  \label{upper perimeter estimate}
  P(\co_r(E);B_s(x))\le \frac{C(n)}{1-(s/r)}\,s^{n-1}\,,\qquad\forall \,s<r\,, x\in\co_r(E)\,.
\end{equation}
\end{lem}

\begin{proof}
  The open set $\co_r(E)^c$ is the union of the open balls $B_r(x)$ such that $B_r(x)\subset E^c$. Thus there exist an at most countable set $I$ such that
  \[
  \co_r(E)^c=\bigcup_{i\in I}\,B_r(x_i)\qquad B_r(x_i)\subset E^c\qquad x_j\ne x_i\ \ \mbox{if $i\ne j$}\,.
  \]
  Let $I_N$ be an increasing sequence of finite subsets of $I$ with $\#I_N=N$. We fix $N$ and for $i\in I_N$ we define
  the spherical region
  \[
  \S_i:=\pa B_r(x_i)\cap\bigcap_{j\in I_N,\,j\ne i } B_r(x_j)^c
  \]
  and the intersection of $B_r(x_i)$ with the cone of vertex $x_i$ over $\S_i$,
  \[
  K_i:=B_r(x_i)\cap \big\{x_i+t\,(x-x_i):x\in\S_i\,,0\le t<1\big\}\,
  \]
 (see Figure \ref{AreaBoundPic}).

\begin{figure}
 \centering
    \includegraphics[scale=0.35]{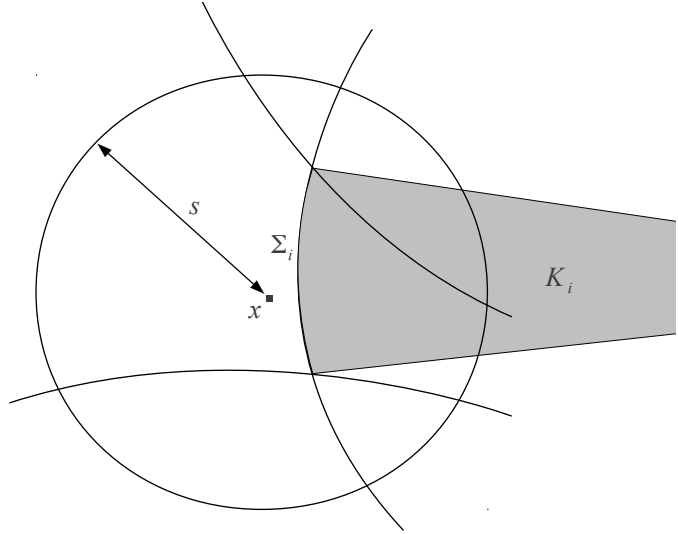}
 \caption{{\small The area of $\Sigma_i$ in $B_s(x)$ is controlled by the area of $\partial B_s(x)$ in $K_i$ and the volume of $B_s(x)$ in $K_i$.}}
\label{AreaBoundPic}
\end{figure}

  Since $x_i\ne x_j$ for each $i\ne j$, we see that
  \begin{equation}
    \label{i dijsoint}
      \H^{n-1}(\S_i\cap\S_j)=|K_i\cap K_j|=0\qquad\forall \,i\ne j\,.
  \end{equation}
  Let $x\in\co_r(E)$ and $s<r$, so that $|y-x_i|> |x-x_i|-s\ge r-s$ for each $i$ and $y\in B_s(x)$.
  By applying the divergence theorem to the vector field $v_i(y)=(y-x_i)/|y-x_i|$ over $K_i\cap B_s(x)$, we find
  \[
  \int_{K_i\cap B_s(x)}\frac{n-1}{|y-x_i|}\,dy=\H^{n-1}(\S_i\cap B_s(x))
  -\int_{K_i\cap\pa B_s(x)}\frac{y-x_i}{|y-x_i|}\cdot\frac{y-x}{|y-x|}\,d\H^{n-1}(y)\,,
  \]
  and thus
  \[
  \H^{n-1}(\S_i\cap B_s(x))\le \H^{n-1}(K_i\cap\pa B_s(x))+(n-1)\,\frac{|K_i\cap B_s(x)|}{r-s}\,.
  \]
  Hence, it follows by \eqref{i dijsoint} that
  \begin{equation}
    \label{from}
  \H^{n-1}\Big(B_s(x)\cap\bigcup_{i\in I_N}\S_i\Big)\le \H^{n-1}(\pa B_s(x))+(n-1)\,|B_1|\,\frac{s^n}{r-s}\le \frac{C(n)}{1-(s/r)}\,s^{n-1}\,.
  \end{equation}
  To conclude the proof,  fix $R>1$ such that $\co_r(E)+B_r \subset B_R$, and set
  \begin{equation}
    \label{FN}
      F_N:=\ov{B_R}\cap \bigcap_{i\in I_N}B_r(x_i)^c\,.
  \end{equation}
  Notice that $F_N$ is a decreasing family of compact sets of finite perimeter, with
  \[
  \co_r(E)=\bigcap_{N\in\N}F_N\qquad \lim_{N\to\infty}|\co_r(E)\Delta F_N|=0\qquad B_R\cap\pa^*F_N=B_R\cap\bigcup_{i\in I_N}\S_i\,,
  \]
  (the last identity modulo $\H^{n-1}$-negligible sets). Since $B_s(x)\cc B_R$ we conclude from \eqref{from} that
  \[
  P(F_N;B_s(x))\le\frac{C(n)}{1-(s/r)}\,s^{n-1}\,,
  \]
  and then deduce \eqref{upper perimeter estimate} by lower semicontinuity of the distributional perimeter.
\end{proof}

\begin{rem}\label{remark FN}
  {\rm Let $E$ be such that $E=\co_r(E)$ for some $r>0$, and let $F_N$ be the compact sets defined in the previous proof.
  Noticing that
  \[
  E+\ov{B_s}=\bigcap_{N\in\N}\,\bigl(F_N+\ov{B_s}\bigr)\,, \qquad \forall\,s>0\,,
  \]
  and since $|E+\ov{B_s}|=|E+B_s|$ we conclude that, for every $s>0$, $|F_N+B_s|\to |E+B_s|$ as $N\to\infty$. In particular, taking $s=r$ and recalling
  \eqref{asymmetry lipschitz} and $|F_N\Delta E|\to 0$ we conclude that
  \begin{equation}
    \label{reduction to FN}
      \a(E)=\lim_{N\to\infty}\a(F_N)\qquad \de_r(E)=\lim_{N\to\infty}\de_r(F_N)\,.
  \end{equation}
  }
\end{rem}

\section{A stability estimate which degenerates for $r$ large}\label{section small} In this section we
present an argument based on the strong quantitative isoperimetric inequality \eqref{FJIsoperimetric}
which leads to a stability estimate which degenerates for $r$ large. We recall that
\[
r_E=\bigg(\frac{|E|}{|B_1|}\bigg)^{1/n}\,.
\]

\begin{thm}\label{thm rconvex 1}
  If $E$ is a bounded $r$-convex set, then
  \[
  C(n)\,\max\bigg\{1,\bigg(\frac{r}{r_E}\bigg)^{n-1}\bigg\}\,\de_r(E)\ge\a(E)^2\,.
  \]
\end{thm}

The following lemma is a preliminary step to the proof of Theorem \ref{thm rconvex 1}. Recall that
\[
\g_E(s)=P(E+B_s)-P(B_{r_E+s})\,,
\]
is non-negative for every $s>0$ thanks to the Brunn-Minkowski and isoperimetric inequalities.

\begin{lem}\label{lemma PerimDeficitExpression}
If $E$ is compact, then $E+B_s$ has finite perimeter for a.e. $s>0$ and
\begin{equation}
  \label{deltarE lower bound}
  \delta_r(E)= c(n)\,\frac{\max\{r_E/r,r/r_E\}}{(r_E+r)^n}\,\int_0^r \g_E(s)\,ds\,,
\end{equation}
with $\g_E(s)\ge c(n)\,|E+B_s|^{(n-1)/n}\,\de_{{\rm iso}}(E+B_s)$ for every $s\ge0$.
\end{lem}

\begin{proof}
By \cite[Lemma 2.1]{carlenmaggi} we have $|E+B_t|=|E|+\int_0^tP(E+B_s)\,ds$ for every $t>0$,
hence $E+B_s$ has finite perimeter for a.e. $s>0$ and
\[
|E+B_r|-|B_{r_E+r}|=\int_0^r \gamma_E(s)\,ds\,.
\]
Therefore
$$
\delta_r(E) = \max\bigg\{\frac{r}{r_E}, \frac{r_E}{r}\bigg\} \frac{1}{|B_{r_E+r}|}\int_0^r \gamma_E(s)\,ds\,.$$
%%%Since $\delta_r(E)<1$ we deduce that
%%%$$
%%% \frac{1}{|B_{r_E+r}|}\int_0^r \gamma_E(s)\,ds\leq 2^n-1,
%%%$$
%%%hence the result follows from the concavity inequality
%%%$$
%%%(1 + t)^{1/n} > 1 + 2^{-n}\,t\,, \qquad \forall\,0  < t < 2^{n}-1\,.
%%%$$
\end{proof}

\begin{proof}
  [Proof of Theorem \ref{thm rconvex 1}] If $E$ is $r$-convex, then $\l\,E$ is $\l\,r$-convex for every $\l>0$. Since $\de_r(E)=\de_{r/r_E}(E/r_E)$ where $|E/r_E|=|B_1|$, up to replacing $E$ by $E/r_E$ it is enough to show that if $r>0$ and $E$ is an $r$-convex set with $|E|=|B_1|$, then
  \begin{equation}
    \label{xyz}
      C(n)\,\max\{1,r^{n-1}\}\,\de_r(E)\ge\a(E)^2\,.
  \end{equation}
  Since $\a(E)<2$, without loss of generality we can assume that
  \[
  \de_r(E)<\frac{\e(n)}{\max\{1,r^{n-1}\}}\le 1
  \]
  for a constant $\e(n)$ to be suitably chosen in the proof. Since $E$ is bounded and $r$-convex we have $E=\co_r(E)$ and thus we can consider the sets $F_N$ introduced in the proof of Lemma \ref{lemma AreaBound} and in Remark \ref{remark FN} (see \eqref{FN}).   We have $|F_N|\ge |B_1|$ and, for $N$ large enough,
  \begin{equation}
    \label{delta FN small}
      \de_r(F_N)<\frac{\e(n)}{\max\{1,r^{n-1}\}}\le 1\,.
  \end{equation}
  Defining
    \[
  \g_N(s):=P(G_s)-P(B_{r_{F_N}+s})\,,\qquad G_s:=F_N+B_s\,,\qquad s\in(0,r)\,,
  \]
 it follows by Lemma \ref{lemma PerimDeficitExpression} applied to $F_N$ that
  \[
  C(n)\,\max\{r,r_{F_N}\}^{n-2}\,\delta_r(F_N) \geq  \frac1r\int_0^r \g_N(s)\,ds\,,\quad \g_N(s)\ge c(n)\,\de_{{\rm iso}}(G_s)\,|G_s|^{(n-1)/n}\quad
  \forall \,s\in(0,r)\,.
  \]
  Since $r_{F_N}\to r_E=1$, there exists $s<\e(n)\,\min\{1,r\}$ such that
  \begin{equation}
    \label{c1}
      C(n)\,\max\{1,r^{n-1}\}\,\de_r(F_N)\ge\g_N(s)\ge c(n)\, \de_{{\rm iso}}(G_s)|G_s|^{(n-1)/n} \ge c(n)\,\beta^*(G_s)^2|G_s|^{(n-1)/n}\,,
  \end{equation}
  where in the last inequality we have used \eqref{FJIsoperimetric}. Up to a translation we assume that
  \[
  \beta^*(G_s)^2=\bigg(\frac{|G_s\Delta B_{r_{G_s}}|}{|G_s|}\bigg)^2+\frac1{|G_s|^{(n-1)/n}}\,\int_{\pa^*G_s}\Big|\nu_{G_s}(y)-\frac{y}{|y|}\Big|^2\,d\H^{n-1}_y\,.
  \]
  Let us notice that
  \[
  \pa^*F_N=\bigcup_{i\in I_N}(B_R\cap\S_i)\cup \Big(\bigcap_{i\in I_N}\ov{B_r(x_i)}^c\cap\pa B_R\Big)
  \]
  with
  \begin{eqnarray*}
  \nu_{F_N}(x)=-\frac{x-x_i}{|x-x_i|}\,,&&\quad\mbox{$\H^{n-1}$-a.e. on $\pa^*F_N\cap\S_i$\,,}
  \\
  \nu_{F_N}(x)=\frac{x}{|x|}\,,&&\quad\mbox{$\H^{n-1}$-a.e. on $\pa^*F_N\cap\pa B_R$}\,.
  \end{eqnarray*}
Recalling that $R>1$, we see that map $T:\pa^*F_N\to\pa^*G_s$ defined as
  \[
  T(x):=x+s\,\nu_{F_N}(x)\qquad\forall \,x\in \pa^*F_N\,,
  \]
  is injective and satisfies
  \[
  \nu_{G_s}(T(x))=\nu_{F_N}(x)\qquad \frac{|T(x)-T(y)|}{|x-y|}\ge \left\{
  \begin{split}
    1-C\,s/r &\qquad\forall \,x,y\in B_R\cap \pa^*F_N
    \\
    1-C\,s/R &\qquad\forall \,x,y\in\pa B_R\cap \pa^*F_N\,.
  \end{split}
  \right .
  \]
  Since $s\le\e(n)\,\min\{1,r\}$ we thus find that the tangential Jacobian $JT:\pa^*F_N\to\R$ of $T$ along $\pa^*F_N$ satisfies
  $JT\ge 1/2$, and thus by the area formula
  \begin{eqnarray}\nonumber
  \int_{\pa^*G_s}\Big|\nu_{G_s}(y)-\frac{y}{|y|}\Big|^2\,d\H^{n-1}_y&\ge&
  \int_{\pa^*F_N}\Big|\nu_{G_s}(T(x))-\frac{T(x)}{|T(x)|}\Big|^2\,JT(x)d\H^{n-1}_x
  \\\label{c2}
  &&\ge \frac12\,
    \int_{\pa^*F_N}\Big|\nu_{F_N}(x)-\frac{T(x)}{|T(x)|}\Big|^2\,d\H^{n-1}_x\,.
  \end{eqnarray}
  Now it is easily seen that if $|e|=1$ and $|z|>2s>0$, then
  \[
  \Big|\frac{z}{|z|}-\frac{z+s\,e}{|z+s\,e|}\Big|\le C(n)\,\min\Big\{\Big|\frac{z}{|z|}-e\Big|,\Big|\frac{z}{|z|}+e\Big|\Big\}\,.
  \]
  Hence, by taking $z=x+s\,\nu_{F_N}(x)=T(x)$ and $e=-\nu_{F_n}(x)$, we  find that
  \begin{equation}
    \label{this last}
      \Big|\frac{x}{|x|}-\frac{T(x)}{|T(x)|}\Big|
  \le C(n)\,   \Big|\frac{T(x)}{|T(x)|}- \nu_{F_N}(x)\Big|\qquad\forall \,x\in\pa^*F_N\,,|x|>3\,s\,.
  \end{equation}
  We now split the argument in two cases.

  \bigskip

  \noindent {\it Case one}: We assume that $|x|>3s$ for $\H^{n-1}$-a.e. $x\in\pa^*F_N$ and for infinitely many values of $N$. In this case, combining \eqref{this last} with \eqref{c1} and \eqref{c2} we obtain
  \[
  C(n)\,\max\{1,r^{n-1}\}\,\de_r(F_N)\ge
  \int_{\pa^*F_N}\Big|\nu_{F_N}(x)-\frac{x}{|x|}\Big|^2d\H^{n-1}_x\ge |F_N|^{(n-1)/n}\,\beta(F_N)^2
  \]
  that is, by \eqref{OscVsAsymmetry},
  \[
  C(n)\,\max\{1,r^{n-1}\}\,\de_r(F_N)\ge |F_N|^{(n-1)/n}\,\a(F_N)^2\,.
  \]
  Since $|F_N| \to |E| = |B_1|$ as $N\to\infty$, we conclude
  by \eqref{reduction to FN} that
  %if $|x|>3s$ for $\H^{n-1}$-a.e. $x\in\pa^*F_N$ for infinitely many values of $N$, then
  \[
  C(n)\,\max\{1,r^{n-1}\}\,\de_r(E)\ge |B_1|^{(n-1)/n}\,\a(E)^2\ge c(n)\,\a(E)^2\,.
  \]
  This completes the discussion of case one.

  \bigskip

  \noindent {\it Case two}: We assume that, for every $N$ large enough,
  \begin{equation}
    \label{case two hp}
      \H^{n-1}(\pa^*F_N\cap B_{3s})>0\,.
  \end{equation}
  Combining \eqref{c1}, \eqref{c2}, and \eqref{this last},  one finds
  \begin{equation}
    \label{case two start}
  \int_{\pa^*F_N}\Big|\nu_{F_N}(x)-\frac{x}{|x|}\Big|^2d\H^{n-1}_x \le 2\,\H^{n-1}(\pa^*F_N\cap B_{3s})
  + C(n)\, \max\{1,r^{n-1}\}\, \de_r(F_N)\,.
    \end{equation}
  Also, it follows by \eqref{case two hp} that there exists $x\in F_N\cap B_{3s}$.
  Thus, since $F_N=\co_r(F_N)$, it follows by \eqref{upper perimeter estimate} that ,  provided $\e(n)$ is small enough,
  \begin{equation}
    \label{combine with claim}
      \H^{n-1}(\pa^*F_N\cap B_{3s})\le \H^{n-1}(\pa^*F_N\cap B_{6s}(x))\le C(n)\,s^{n-1}\,.
  \end{equation}
  We now claim that \eqref{case two hp} implies
  \begin{equation}
    \label{claim}
      \int_{\pa^*G_s}\Big|\nu_{G_s}(x)-\frac{x}{|x|}\Big|^2d\H^{n-1}_x\ge c(n)\,s^{n-1}\,.
  \end{equation}
  Notice that, if we can prove \eqref{claim},
  then combining that estimate with \eqref{c1}, \eqref{combine with claim}, and \eqref{case two start}, we get
  \[
  \int_{\pa^*F_N}\Big|\nu_{F_N}(x)-\frac{x}{|x|}\Big|^2d\H^{n-1}_x\le
  C(n)\,\max\{1,r^{n-1}\}\,\de_r(F_N)\,,
  \]
  which allows us to conclude as in case one. Hence, to conclude the proof is enough to
prove \eqref{claim}.

To this aim we first notice that, setting for
   the sake of brevity $\sigma:=r_{G_s}$ (so that $\sigma\ge r_{F_N}\ge r_E=1$), then \eqref{delta FN small} and \eqref{c1} imply that
  \begin{equation}
    \label{Gs big in BR}
  |B_\sigma\setminus  G_s|\le|G_s\Delta B_\sigma|\le \sqrt{\e(n)}\,|B_\sigma|\,.
  \end{equation}
  Next we notice that there exists a ball $B_s(z_0)$ such that
  \begin{equation}
    \label{Bsz}
    B_s(z_0)\subset (G_s)^c\cap B_{6\,s}\,.
  \end{equation}
  Indeed, \eqref{case two hp} implies the existence of $x_0\in\pa^*F_N\subset\pa F_N$ with $|x_0|<3s$. Since $F_N$ is $r$-convex there exists $B_r(y_0)\subset (F_N)^c$ such that $|x_0-y_0|=r$. Since $B_{r-s}(y_0)\subset (G_s)^c$, $|y_0-x_0|=r$, and $|x_0|<3s$, we can easily find $z_0$ such that \eqref{Bsz} holds.

  Now let $\tau\in(0,1)$ be a parameter to be fixed later on and let $e=z_0/|z_0|$ if $z_0\ne 0$, or $e=e_1$ if $z_0=0$. Let
  us consider the spherical cap
  \[
  \S_\tau:=\Big\{x\in\pa B_s(z_0):\mbox{$x=z_0+s\,v$ with $v\cdot e=\cos\a$ for some $|\a|<\tau$}\Big\}\,.
  \]
  Given $x\in\S_\tau$ let
  \[
  t(x):=\inf\Big\{t>0:z_0+t\,\frac{x-z_0}s\in\pa^*G_s\Big\}\,,
  \]
  so that $t(x)\in(s,\infty]$, and let
  \[
  \S_\tau^*:=\Big\{x\in\S_\tau:t(x)<\sigma-6s\Big\}\,.
  \]
  In this way the map $\psi:\S_\tau^*\to\R^n$ defined as
  \[
  \psi(x):=z_0+\frac{t(x)}{s}(x-z_0)\qquad x\in\S_\tau^*
  \]
  satisfies $\psi(\S_\tau^*)\subset B_\sigma\cap\pa^*G_s$ with
  \begin{equation}
    \label{aaa}
      \nu_{G_s}(\psi(x))\cdot \frac{x-z_0}{s}\le 0\,,\qquad\forall \,x\in\S_\tau^*\,,
  \end{equation}
  and a Taylor's expansion gives
  \begin{equation}
    \label{bbb}\Big|\frac{\psi(x)}{|\psi(x)|}-\frac{x-z_0}{s}\Big|\le C\,\tau\qquad\forall \,x\in\S_\tau^*\,.
  \end{equation}
  Hence, if $y\in\S_\tau^*$, then by \eqref{aaa} and \eqref{bbb}
  \[
  \frac12\Big|\nu_{G_s}(y)-\frac{y}{|y|}\Big|^2=1-\nu_{G_s}(\psi(x))\cdot\frac{\psi(x)}{|\psi(x)|}\ge 1-C\tau\ge \frac12\,
  \]
  (see Figure \ref{OscBoundPic}), therefore
    \begin{equation}
    \label{us}
      \int_{\pa^*G_s}\Big|\nu_{G_s}(y)-\frac{y}{|y|}\Big|^2\,d\H^{n-1}_y\ge\H^{n-1}(\S_\tau^*)\,.
  \end{equation}
 \begin{figure}
 \centering
    \includegraphics[scale=0.35]{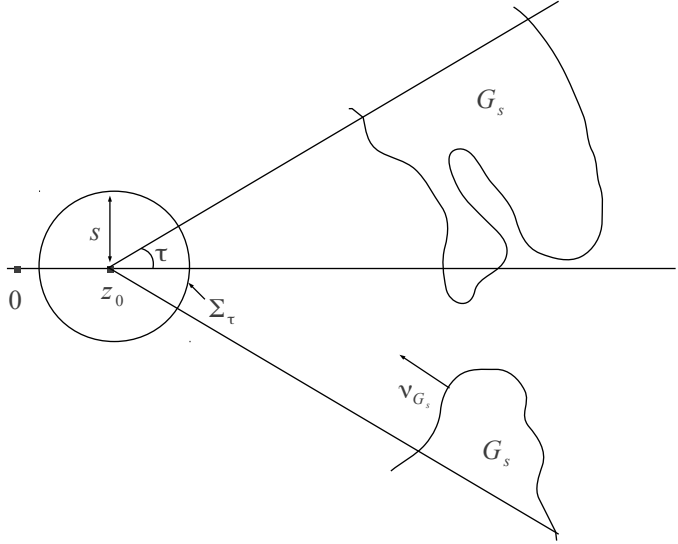}
 \caption{{\small On the part of $\partial G_s$ that is accessible to rays from $z_0$ through $\Sigma_{\tau}$, the normal $\nu_{G_s}$ is far from $x/|x|$.}}
\label{OscBoundPic}
\end{figure}

On the other hand we have
  \[
  Z:=\Big\{z_0+t\,\frac{x-z_0}s:x\in\S_\tau\setminus\S_\tau^*\,,t\in(0,S-6s)\Big\}\subset B_\sigma\setminus G_s\,,
  \]
  so it follows by \eqref{Gs big in BR} that
  \[
  C(n)\,\sqrt{\e(n)}\,\sigma^n\ge|Z|=\int_0^{\sigma-6s}\,\rho^{n-1}\,\H^{n-1}\Big(\frac{\S_\tau\setminus\S_\tau^*}s\Big)\,d\rho=
  \frac{(\sigma-6s)^n}{n\,s^{n-1}}\H^{n-1}(\S_\tau\setminus\S_\tau^*)
  \]
  that is $\H^{n-1}(\S_\tau\setminus\S_\tau^*)\le C(n)\,\sqrt{\e(n)}\,s^{n-1}$. Since $\H^{n-1}(\S_\tau)\ge c(n,\tau)\,s^{n-1}$, choosing $\e$ small enough we conclude that $\H^{n-1}(\S_\tau^*)\ge c(n)\,s^{n-1}$,
  that combined with \eqref{us}
  concludes the proof of \eqref{claim}.
\end{proof}

\section{A stability estimate for $r/r_E$ large}\label{section large} We now address the case when $r/r_E$ is large.

\begin{thm}\label{thm rlarge}
  There exists $C(n)$ such that if $E$ is a measurable set with $0<|E|<\infty$ and $r\ge C(n)\,r_E$, then
  \[
  C(n)\,\de_r(E)\ge\a(E)^2\,.
  \]
\end{thm}

The argument exploits the notion of sets of positive reach and the corresponding {\it Steiner's formula}. Let us recall that if $r>0$ and $E$ is a closed subset of $\R^n$, then $E$ has {\it positive reach $r$} if for every $x$ with $\dist(x,E)<r$ there exists a unique closest point to $x$ in $E$. This property allows one to exploit the area formula to deduce that $P(E+B_s)$ is a polynomial for $s\in[0,r]$.

\begin{thm}[Federer \cite{FedererCURVMEASURES}]\label{SteinerFormula}
 If $E$ has positive reach $r$, then $s\mapsto P(E+B_s)$ is a polynomial of degree at most $n-1$ on the interval $[0,r]$.
\end{thm}

A second tool used in our argument is the following elementary lemma.
\begin{lem}\label{PolyLem}
If $p$ is a non-negative polynomial on $[0,1]$, then
$$
\int_0^1 p(x)\,dx \geq c\,p(0)
$$
where $c$ is a positive constant depending only on the degree of $p$.
\end{lem}

\begin{proof} Let us assume without loss of generality that $p(0)=1$ and set $N={\rm deg}(p)$. Let $\a\ge 0$ be the largest slope such that $\a\,x\le p(x)$ for every $x\in[0,1]$. Since $p(x)-\a\,x$ and $p(x)$ are both non-negative polynomials on $[0,1]$ with same value at $x=0$, we can replace $p(x)$ with $p(x)-\a\,x$. In doing so we gain the information that our polynomial has either a zero at $x=1$, or a zero of order at least two at $x=a$ for some $a\in(0,1)$. In particular, either $p(x)=q(x)\,(1-x)$ or $p(x)=q(x)\,(x-a)^2$, where $q$ is a non-negative polynomial in $[0,1]$ with $q(0)\ge 1$ and degree strictly less than the degree of $p$. By iterating the procedure we reduce ourselves to the case where,
for some $0 \leq k \leq N/2$,
$$
p(x) = c\,(1-x)^{N-2k}\,\prod_{i=1}^k(x-a_i)^2\,,
$$
where $\{a_i\}_{i=1}^k\subset(0,1)$ and $c\ge 1$ (since $p(0)\ge 1$). In particular,
\[
\int_0^1\,p(x)\,dx\ge \min_{(b_1,...,b_N) \in [0,1]^N} f(b_1,...,b_N)\,, \qquad \mbox{where} \quad f(b_1,...,b_N) := \int_{0}^1 \prod_{i=1}^N |x-b_i|\,dx\,.
\]
Clearly $f$ is continuous and strictly positive on $[0,1]^N$. By compactness,
\[
\min_{(b_1,...,b_N) \in [0,1]^N} f(b_1,...,b_N) > c(N)\,,
\]
and the proof is complete.
\end{proof}

We are now ready to prove Theorem \ref{thm rlarge}.

\begin{proof}
  [Proof of Theorem \ref{thm rlarge}] We can directly assume that $E$ is $r$-convex with
  \[
  |E|=|B_1|\,,\qquad \de_r(E)<1\,,\qquad r\ge C(n)\ge 1\,.
  \]
  We claim that
  \begin{equation}
    \label{final claim}
    \mbox{$E$ has positive reach $1$}\,.
  \end{equation}
  This claim allows one to quickly conclude the proof. Indeed, it follows by Theorem \ref{SteinerFormula} that,
   for $s < 1$, the perimeter deficit $\gamma_E(s)$ is a nonnegative polynomial of
 degree at most $n-1$, so we can apply Lemma \ref{PolyLem} to conclude that
 $$|E + B_1| - |B_2| = \int_0^1 \gamma_E(s)\,ds \geq c(n)\,\gamma_E(0)\,.$$
 In particular, applying \eqref{sharp quantitative isoperimetric inequality} we obtain
 $$|E + B_1| \geq |B_2| + c(n)\,\alpha(E)^2 \geq |B_{2 + c(n)\,\alpha(E)^2}|,$$
 since $\alpha(E) \leq 2$. Since $r > 1$ by assumption, it follows from the Brunn-Minkowski inequality that
 $$
 |E + B_r|=|(E+B_1)+B_{r-1}|\ge(|E+B_1|^{1/n}+|B_{r-1}|^{1/n})^n \geq |B_{1 + r + c(n)\,\alpha(E)^2}|\,$$
hence
 $$\delta_r(E) \geq c(n)\,\frac{r}{1+r}\,\alpha(E)^2 \geq c(n)\,\alpha(E)^2,$$
 which is the desired inequality. Hence, we are left to prove \eqref{final claim}.
 Before doing so, we first make some comments about the argument we just showed.

 \begin{rem}{\rm
 The proof above works for any set $E$ with reach $\min\{1,r\}$, for any $r>0$. In particular, it proves the main theorem in the case that $E$ is convex.
 In the case when $r$ is small one has that
 $$C(n)\,\delta_r(E) \geq \frac{1}{r}\int_{0}^r \gamma_E(s)\,ds.$$
 Since the integrand is a positive polynomial for $s < r$,
 the result follows immediately from Lemma \ref{PolyLem} and inequality \eqref{sharp quantitative isoperimetric inequality}.

Note however that, given a bounded set $E$,
for small $r$ we cannot affirm that $\co_r(E)$ has reach $r$. Take for example $B_2$ minus two small balls of radius $2r$ whose centers are at distance $4r-\delta$ with $\delta \ll r$. Then this set coincides with its
 $r$-envelope, but has reach that goes to $0$ as  $\delta\to 0$.}
\end{rem}

\begin{rem}\label{ConvexCaseRLarge}
 {\rm Our technique also gives an alternative proof of sharp stability for Brunn-Minkowski in the case that $E$ and $F$ are both convex.
 In this case one uses the existence of a Steiner polynomial for a certain weighted perimeter of $E + sF$ (see \cite{BuragoZalgallerBOOK}), and the sharp
 quantitative anisotropic isoperimetric inequality \cite{FigalliMaggiPratelliINVENTIONES}.}
\end{rem}

 We now prove \eqref{final claim}. We achieve this in five steps.

  \bigskip

  \noindent {\it Step one:} We prove that $\diam(E)\le R(n)$ for some $R(n)$. Indeed given two points $x,y\in E$, we find
  \[
  |E + B_r| \geq |B_r(x)\cup B_r(y)| \geq |B_r|\min\bigg\{1 + c_0(n)\,\frac{|x-y|}r, 2\bigg\}\,.
  \]
  Since $r\ge 1$ and thus $r/(1+r)\ge 1/2$, we get
 \begin{eqnarray}\label{step one}
 \delta_r(E)&\ge&\max\bigg\{r,\frac{1}{r}\bigg\}\,\bigg(\frac{|E+B_r|^{1/n}}{(1+r)|B_1|^{1/n}}-1\bigg)
 \\\nonumber
 &&\ge
 \frac12\left(\min\bigg\{r\,\Big(1 + c_0(n)\,\frac{|x-y|}r\Big)^{1/n}, 2^{1/n}\,r\bigg\} -1-r\right)\,.
 \end{eqnarray}
 If the minimum in the right hand side was achieved by $2^{1/n}r$ then we would find $2\,\de_r(E)\ge (2^{1/n}-1)\,r-1\ge 2$
 provided $r\ge C(n)$ for $C(n)$ large enough,
 which is impossible since $\de_r(E)<1$. Thus we must have that
 \[
 \Big(1 + c_0(n)\,\frac{|x-y|}r\Big)^{1/n}\le 2^{1/n}\,,\quad\mbox{that is}\quad c_0(n)\,\frac{|x-y|}r\le 1\,.
 \]
 Since $(1+s)^{1/n}\ge 1+c(n)\,s$ for every $s\in(0,1)$, we deduce that
 \[
 r\,\Big(1 + c_0(n)\frac{|x-y|}r\Big)^{1/n}-1-r\ge c(n)\,|x-y|-1\,,
 \]
 hence it follows by \eqref{step one} that
 $c(n)\,|x-y|\le 1+2\,\de_r(E)\le 3$, as desired.

 \bigskip

 \noindent {\it Step two}: By step one it follows that, up to a translation, $E\subset B_{R(n)}$. We now prove that (after possibly another translation)
 \[
 B_{r(n)}\subset E
 \]
 for some $r(n)>0$.

 Let $\co(E)$ denote the convex hull of $E$.
 By John's lemma there exists an affine transformation $L:\R^n\to\R^n$ such that
 \[
 B_1\subset L(\co(E))\subset B_n\,.
 \]
 Since $\co(E)\subset B_{R(n)}$ and $|E|=|B_1|\le|\co(E)|$ we deduce that $\Lip L\,,\Lip\,L^{-1}\le C(n)$. In particular,
 up to a translation, we deduce that
 \[
 B_{r(n)}\subset \co(E)
 \]
 for some $r(n)>0$. We now claim that if $r\ge C(n)$ for $C(n)$ large enough, then $B_{r(n)/2}\subset E$ too.

 Indeed, if not, there exists
 $x\in E^c\cap B_{r(n)/2}$. Since $E$ is $r$-convex, this means that $x\in B_r(y)$ for some $B_r(y)\subset E^c$,
 which implies that  (after possibly replacing $R(n)$ by $2R(n)$)
  \begin{figure}
 \centering
    \includegraphics[scale=0.3]{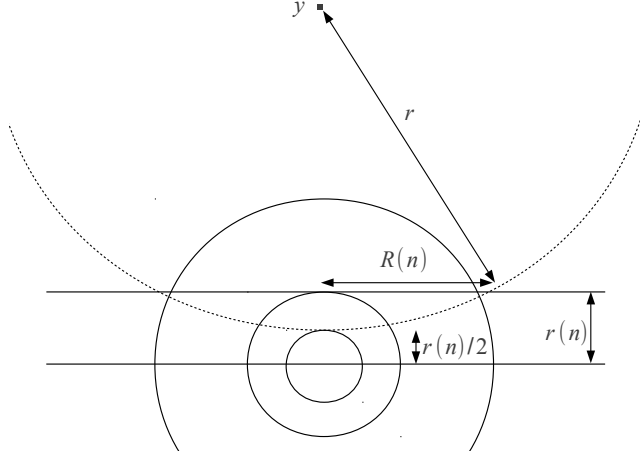}
 \caption{{\small If $r$ is large and an exterior ball $B_r(y)$ intersects $B_{r(n)/2}$, then $B_{r(n)}$ is not in the convex hull of $E$.}}
\label{InteriorBallPic}
\end{figure}
 \begin{equation}
   \label{hey}
 E\subset B_{R(n)}\setminus B_r(y)\,,\qquad\mbox{where}\quad |y|\le r+\frac{r(n)}2\,.
 \end{equation}
 Since $r\ge C(n)$ and $|y|\le r+r(n)/2$, we can pick $C(n)$ large enough with respect to $r(n)$ and $R(n)$ to ensure that
 \[
 B_{R(n)}\setminus B_r(y)\subset\Big\{z\in\R^n:z\cdot\frac{y}{|y|}\le \frac{2\,r(n)}3\Big\}\,.
 \]
 Recalling \eqref{hey}, this implies that
 \[
 \co(E)\subset\Big\{z\in\R^n:z\cdot\frac{y}{|y|}\le \frac{2\,r(n)}3\Big\}\,,
 \]
 against $B_{r(n)}\subset\co(E)$ (see Figure \ref{InteriorBallPic}).

 \bigskip

 \noindent {\it Step three}: We claim that there exists $\de(n)\in(0,1)$ with the following property: for every two-dimensional plane $\Pi$ through the origin and each
 $x\in\Pi\cap\pa E$, there exists a disk of radius $\de(n)\,r$ contained in $\Pi\cap E^c$ whose boundary contains $x$.

 Indeed, let $B_r(z)$ be an exterior tangent ball to $E$ that touches $x$. Choose coordinates so that $\Pi$ is the $e_1,\,e_2$ plane.
 By rotations in $\Pi$ and then in its orthogonal complement, we may assume that
 $z = z_1e_1 + z_ne_n$, with $z_i \geq 0$. Since $B_r(z) \cap B_{r(n)} = \emptyset$ (by Step two) we have
 $$z_1^2 + z_n^2 \geq (r+r(n))^2.$$
 On the other hand, since $x \in B_{R(n)}$ we know $\partial B_r(z)$ intersects the $x_1$ axis at some point $\alpha e_1$ with $r(n) \leq \alpha \leq R(n)$. We thus have
 $$(z_1 - \alpha)^2 + z_n^2 = r^2.$$
 Combining these we obtain
 $$z_1 \geq \frac{2r\,r(n) + r(n)^2 + \alpha^2}{2\alpha} > \frac{r(n)}{R(n)}r.$$
 Note that $B_r(z) \cap \Pi$ is a disc of radius $z_1-\alpha$ centered on the $x_1$ axis, and by the above we have
 $$z_1 - \alpha \geq r\left(\frac{r(n)}{R(n)} - \frac{R(n)}{r}\right) > \frac{r(n)}{2R(n)}\,r$$
 provided $r > 2R(n)^2/r(n)$, which proves our claim (see Figure \ref{BallSlicePic}).

\begin{figure}
 \centering
    \includegraphics[scale=0.3]{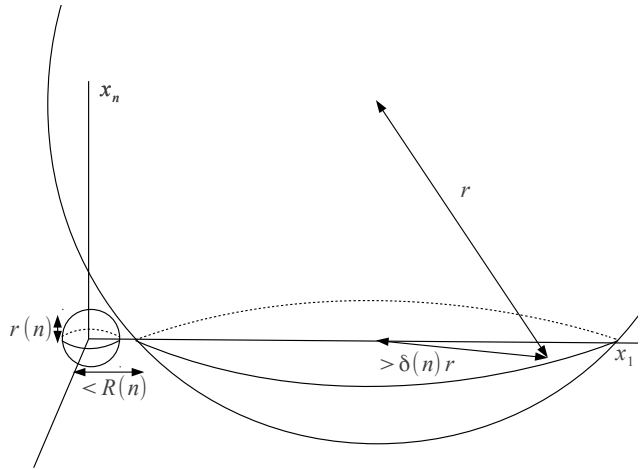}
 \caption{{\small When we restrict to a plane, $E$ has exterior tangent circles of radius comparable to $r$.}}
\label{BallSlicePic}
\end{figure}

 \bigskip

 \noindent {\it Step four}: We show that $\partial E \cap \Pi$ is a radial graph for any two-plane $\Pi$ through the origin. Indeed, recall that
 $B_{r(n)} \subset E$. Follow a ray from the origin in $\Pi$ until the first time it hits $\partial E$ at some point $x \in B_{R(n)}$. By the previous step we know that
 there is an exterior tangent disc in $\Pi$ of radius $\delta(n) r$ whose boundary contains $x$. Using that this disc does not intersect $B_{r(n)} \cap \Pi$
 and similar arguments to those in the previous step, one concludes that the radial line segment from $x$ to $\partial B_{R(n)}$ is in $E^c$ if $r$
 is sufficiently large, and since $E \subset B_{R(n)}$ the claim is established.

 \bigskip

 \noindent {\it Step five}: We finally prove \eqref{final claim}.
  In view of the previous steps we may take $r$ large so that, after a translation, $B_{r(n)} \subset E \subset B_{R(n)}$,
  and $\partial E$ restricted to any two-plane is a radial graph with external tangent circles of radius $2$.

 Assume by way of contradiction that $E$ does not have reach $1$. Then there is some exterior tangent ball of radius less than $1$ touching
 $\partial E$ at two points $y_1$ and $y_2$. Let $\Pi$ be the two-plane containing $y_i$ and the origin,
 and choose coordinates so that $\Pi$ is the $e_1,\,e_2$ plane.
 Then there is an exterior tangent disc
 of radius less than $1$ touching $\partial E \cap \Pi$ at $y_1$ and $y_2$. Denote by $D_t(x)$ the disk in $\Pi$ of radius $t$ centered at $x$.
Up to a rotation in $\Pi$, we can assume that both $y_1$ and $y_2$ touch an exterior tangent disc $D_{t_0}(\beta e_1)$, with $t_0 \leq 1$ and $\beta > t_0 + r(n)$
 (since $B_{r(n)} \subset E$).

 Now, let $S$ be the sector in $\Pi$ bounded by the rays from the origin through $y_i$,
and
 let $C_s$ be the radial graph given by the left part of $\partial D_{t_0}((t_0 + s)e_1) \cap S$, for $0 \leq s \leq \beta-t_0$.
 Note that $C_s \subset E \cap \Pi$ for all $s$ small (since $B_{r(n)} \subset E$). Furthermore, the endpoints of $C_s$ on $\partial S$ are
 in $E$ for all $0 \leq s < \beta - t_0$ since $\partial E \cap \Pi$ is a radial graph.
 Since $D_{t_0}(\beta e_1)$ is exterior to $E \cap \Pi$, we can increase $s$ until $C_s$ first touches $\partial E \cap \Pi$ for some
 $0 < s_0 \leq \beta - t_0$. (In particular, $C_s \subset E\cap \Pi$ for all $0 \leq s < s_0$.)
 Hence, this proves the existence of a point $y_3 \neq y_1,\,y_2$ which belongs to $C_{s_0} \cap \partial E$.

 Since $y_3 \in \partial E$, we know that there is an exterior tangent disc of radius $2$ whose boundary contains $y_3$.
 On the other hand, since $t_0 \leq 1$, any such disc will contain points
 in $C_s$ for some $s < s_0$, a contradiction that concludes the proof (see Figure \ref{PositiveReachPic}).

 \begin{figure}
 \centering
    \includegraphics[scale=0.35]{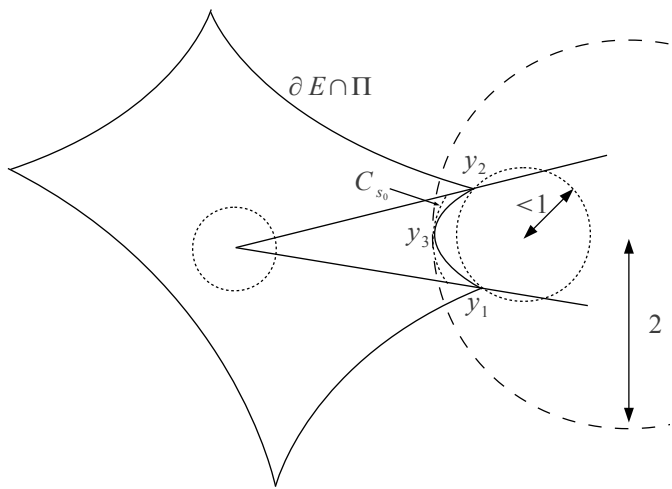}
 \caption{{\small If on some two-plane the radial graph $\partial E$ does not have reach $1$, we get a contradiction for $r$ large.}}
\label{PositiveReachPic}
\end{figure}

\end{proof}

%%%%%%%%%%%%%%%%%%%%%%%%%%%%%%%%%%%%%%%%%%%%%%%%%%%%%%%%%%%

%%%%%%%%%%%%%%%%%%%%%%%%%%%%%%%%%%%%%%%%%%%%%%%%%%%%%%%%%%%%%%%%%%%%%%%%

\section*{Acknowledgments}
A. Figalli was supported by NSF Grants DMS-1262411 and DMS-1361122.
F. Maggi was supported by NSF Grants DMS-1265910 and DMS-1361122.
C. Mooney was supported by NSF grant DMS-1501152. C. Mooney would like to thank L. Caffarelli for a helpful conversation.

%%%%%%%%%%%%%%%%%%%%%%%%%%%%%%%%%%%%%%%%%%%%%%%%%%%%%%%%%%%%%%%%%%%%%%%%%%%%%%%%%%%
\bibliography{references}
\bibliographystyle{alpha}
%\bibliographystyle{is-alpha}
%%%%%%%%%%%%%%%%%%%%%%%%%%%%%%%%%%%%%%%%%%%%%%%%%%%%%%%%%%%%%%%%%%%%%%%%%%%%%%%%%%%%
%%%%%%%%%%%%%%%%%%%%%%%%%%%%%%%%%%%%%%%%%%%%%%%%%%%%%%%%%%%%%%%%%%%%%%%%%%%%%%%%%%%%%%%%%%%%%%%%%%%%%%%%%%%%

\end{document}